%% file: short_paper.tex
\documentclass[3p]{elsarticle}
\usepackage{mathtools}

\usepackage{amssymb}
\usepackage[english,francais]{babel}
\usepackage{amsthm}
\usepackage{amsmath}
\usepackage{amsfonts}
\usepackage{color}
\usepackage{pgfplots}
\usepackage{dsfont}

\usepackage{enumitem}
\usepackage[]{fullpage}

\newcommand{\norm}[1]{\ensuremath{\left\lVert#1\right\rVert}}
\newcommand{\abs}[1]{\ensuremath{\left\lvert#1\right\rvert}}
\DeclareMathOperator{\R}{\mathbb{R}}
\newcommand{\set}[1]{\ensuremath{\left\lbrace#1\right\rbrace}}

\newtheorem{theorem}{Theorem}
\newtheorem{corollary}{Corollary}
\newtheorem{definition}{Definition}
\journal{}
\date{}

\makeatletter
\def\ps@pprintTitle{%
	\let\@oddhead\@empty
	\let\@evenhead\@empty
	\def\@oddfoot{}%
	\let\@evenfoot\@oddfoot}
\makeatother

\begin{document}
\begin{frontmatter}
\title{Exponential decay of the resonance error in numerical homogenization via parabolic and elliptic cell problems}
\author[1]{Assyr Abdulle}
\ead{assyr.abdulle@epfl.ch}
\author[1]{Doghonay Arjmand}
\ead{doghonay.arjmand@epfl.ch}
\author[1]{Edoardo Paganoni}
\ead{edoardo.paganoni@epfl.ch}
\address[1]{ANMC, \'Ecole Polytechnique F\'ed\'erale de Lausanne, 1015 Lausanne, Switzerland}

\begin{abstract}
\selectlanguage{english}
This paper presents two new approaches for finding the homogenized coefficients of multiscale elliptic PDEs. 
Standard approaches for computing the homogenized coefficients suffer from the so-called resonance error, originating from a mismatch between the true and the computational boundary conditions. Our new methods, based on solutions of 
parabolic and elliptic cell-problems, result in an exponential decay of the resonance error.  
\vskip 0.5\baselineskip

\selectlanguage{francais}
\noindent{\bf R\'esum\'e} \vskip 0.5\baselineskip \noindent
{\bf D\'ecroissance exponentielle de l'erreur de r\'esonance en homog\'en\'eisation numerique via des probl\`emes de cellules paraboliques et elliptiques.}
Cette note pr\'esent deux nouvelles approches pour trouver les coefficients homog\'en\'eis\'es des EDP elliptiques multi-\'echelles.
Les approches standard pour calculer les coefficients homog\'en\'eis\'es souffrent de ce que l'on appelle l'erreur de r\'esonance, qui d\'ecoule d'une inad\'equation entre les vraies conditions aux limites et celles computationelles. 
Nos nouvelles m\'ethodes, bas\'ees sur des solutions des probl\`emes de cellules paraboliques et elliptiques, entra\^inent une d\'ecroissance exponentielle de l'erreur de r\'esonance.

\end{abstract}
\end{frontmatter}
\selectlanguage{english}
\section{Introduction}
\label{sec:intro}
\input{tex/intro}
	
\section{New algorithms for computing the homogenized tensor}
\label{sec:parabolic model}
\input{tex/proposed_approaches}
	\subsection{A parabolic approach}
	\label{sec:parabolic approach}
	\input{tex/parabolic_approach}
	\subsection{A modified elliptic approach}
	\label{sec:modified elliptic approach}
	\input{tex/modified_elliptic_approach}

\section{Main results}
	\label{sec:main}
	\subsection{Equivalence between the standard elliptic and the parabolic formulations}
	\label{sec:equivalence}
	\input{tex/equivalence_rescaled}
	\subsection{Exponential convergence of the parabolic approach \eqref{eq:new a0 filtered}}
	\input{tex/Exponential_Convergence_Parabolic}
	\subsection{Exponential convergence of the modified elliptic approach \eqref{eq:new_elliptic_upscaling_formula}}
	\input{tex/Exponential_Convergence_Modified_Elliptic}
	
\section{Numerical validation}
\input{tex/numerical_validation}

\end{document}

%% file: tex/intro.tex
We consider the numerical homogenization of multiscale elliptic partial differential equations (PDEs) of the form 
\begin{equation}
\label{eq:model problem}
\left\lbrace
\begin{aligned}
-\nabla \cdot \left( a^{\varepsilon}(x) \nabla u^{\varepsilon} \right) &= f & \quad  & \text{in }\Omega \subset \mathbb{R}^{d} \\
u^{\varepsilon} &= 0 & \quad & \text{on } \partial \Omega,
\end{aligned}
\right.
\end{equation}
where $a^{\varepsilon} \in \left[ L^{\infty} (\Omega) \right]^{d \times d}$ is symmetric, uniformly elliptic and bounded, and ${\varepsilon \ll |\Omega|^{1/d} = O(1)}$ is the wave-length of the small scale variations in the medium. A direct numerical approximation of $u^{\varepsilon}$ by standard finite element/difference methods is prohibitively expensive as the $\varepsilon$-scale variations need to be resolved on the whole computational domain $\Omega$. Homogenization theory aims at finding an effective coefficient $a^{0}$ (or solution $u^{0}$) such that $-\nabla \cdot \left( a^{0}(x) \nabla u^{0} \right) = f$ describes the coarse-scale behaviour of \eqref{eq:model problem}. The coefficient $a^{0}$ (and hence the solution) is no more oscillatory, and a standard solver may be directly applied to the homogenized system once $a^{0}$ is determined. Explicit representations for $a^{0}$ are available only in a few cases, such as periodic microstructures or stationary ergodic random materials. For example, when the medium is such that $a^{\varepsilon}(x) = a(x/\varepsilon)$, and $a$ is a ${K:=[-1/2,1/2]^{d}}$-periodic function, then $a^0$ is given by  
\begin{equation}\label{eq: upscaling formula}
a^0_{ij} =  \frac{1}{\abs{K}} \int_{K}\left(\mathbf{e}_i + \nabla \chi^i(y)\right) \cdot a(y) \left( \mathbf{e}_j + \nabla \chi^j (y) \right) \,dy,
\end{equation}
where $\chi^i$ is a $K$-periodic solution of the so-called \emph{corrector problem}, see \cite{BLP78,CiD99,JKO94}:
\begin{equation}
\label{eq:micro_problem_periodic_domain}
-\nabla\cdot\Big( a(y) \left( \nabla\chi^i + \mathbf{e}_i \right) \Big) = 0 \quad  \text{in } K.
\end{equation}
In several situations of interest, e.g. when the period of $a$ is not known or when  $a$ is quasi-periodic or stochastic, equations \eqref{eq: upscaling formula} and \eqref{eq:micro_problem_periodic_domain} have to be posed over the whole $\mathbb{R}^d$.
In this case, we write ${a^{0}= \lim_{R \to +\infty} a^{0,R}}$, where 
\begin{equation}\label{eq:upscaling_formula_bounded_domain}
a^{0,R}_{ij} = \frac{1}{\abs{K_R}} \int_{K_R}\big(\mathbf{e}_i + \nabla \psi^i_{R}(y)\big) \cdot a(y) \big( \mathbf{e}_j + \nabla \psi^j_{R}(y)\big) \,dy,
\end{equation}
the domain $K_{R}:=\left[-\frac{R}{2},\frac{R}{2}\right]^{d}$, and $\psi^i_{R}$ solves the Dirichlet problem 
\begin{equation}
\label{eq:micro_problem_bounded_domain}
\left\lbrace
\begin{aligned}
-\nabla\cdot\Big( a(y) \left( \nabla\psi^i_{R} + \mathbf{e}_i \right)\Big) &= 0 & \quad  & \text{in } K_R \\
\psi^i_{R} &= 0 & \quad & \text{on } \partial K_R.
\end{aligned}
\right.
\end{equation}
In practice, the value of $a^{0,R}$ can be computed only for finite values of $R$, and an error occurs due to the mismatch on the boundary $\partial K_R$ between the values of $\psi^i_{R}$ and $\chi^i$. This error will then propagate into the domain $K_R$ and deteriorate the accuracy of the approximation $a^{0,R}$. 
It is well-known that if $a$ is $K$-periodic and $R$ is not integer, then ${\norm{  a^{0,R} - a^0 }_{F} \le C R^{-1}}$, see \cite{AEE12,EMZ05}. A similar result exists also for stationary ergodic random coefficients, both in continuous \cite{BoP04} and discrete \cite{GlO12} settings. This first order resonance error dominates all other discretization errors in modern multiscale methods and, therefore, better approximation techniques with reduced resonance errors are needed. 

In order to reduce the resonance error, previous approaches improved the prefactor (but not the convergence rate) \cite{YuE07}, or gave second order rates in $1/R$ \cite{BlL10}, or fourth order in the asymptotic limit for large values of $R$ \cite{Glo11}. Another strategy results in arbitrary orders in $1/R$, but at the cost of solving a computationally expensive wave equation \cite{AR16}. 

This paper, inspired by \cite{Mou18}, presents two strategies based on parabolic and elliptic corrector problems which have exponentially decaying boundary errors at a cost comparable to the one of solving the classical elliptic model.
%
%

%% file: tex/proposed_approaches.tex
It can be seen that the computation of homogenized coefficients is linked to the average of oscillatory functions, as formula \eqref{eq:upscaling_formula_bounded_domain} shows. A naive averaging of a $K$-periodic function $f$ over $K_R$ converges to the mean value ${\frac{1}{|K|}\int_{K} f(y) \; dy}$ with a first order accuracy in $1/R$. To improve this accuracy to arbitrarily high rates, a set of smooth averaging filters can be used, see \cite{Glo11}. 
%
%
%
\begin{definition} \label{def:Filters}
	We say that a function $\mu:[-1/2,1/2]\mapsto\R_+$ belongs to the space $\mathbb{F}_{q}$, $q\ge 1$, if:  
	\begin{enumerate}[label=\roman*)]
		\item $\displaystyle \mu\in C^q([-1/2,1/2]) \cap W^{q+1,\infty} ((-1/2,1/2))$;
		\item $\displaystyle \mu^{(k)}(-1/2) = \mu^{(k)}(1/2) = 0, \, \forall k \in\left\lbrace0,\dots,q-1\right\rbrace$;
		\item $\displaystyle \int_{-\frac{1}{2}}^{\frac{1}{2}} \mu\left(y\right) dy = 1$.
	\end{enumerate}
	For $q=0$, we define $\mu\in\mathbb{F}_0$ as ${\mu(y) = \mathds{1}_{[-1/2,1/2]}}$, where $\mathds{1}_{I}$ is the characteristic function on the interval $I$.
	
	We say that a function ${\mu_L: K_L:=[-L/2,L/2]^d\subset \mathbb{R}^d \to \mathbb{R}_{+}}$, with $L>0$, belongs to the space $\mathbb{F}_{q}(K_L)$ if $\mu_{L} (y) = \frac{1}{L^d} \prod_{i=1}^{d} \mu \left( \frac{y_i}{L} \right)$, where $\mu \in \mathbb{F}_{q}$ and $y_i$ is the coordinate along the $i$-th direction.
\end{definition}

%% file: tex/parabolic_approach.tex
In this section we introduce a numerical homogenization scheme based on the solution of parabolic differential equations, as proposed in the discrete setting \cite{Mou18}, since the parabolic Green's function decays exponentially in space. This yields to a reduced influence of mismatching boundary values on the corrector functions.
%
The new cell problems are defined by
\begin{equation} \label{eq:parabolic dirichlet problem}
	\left\lbrace
	\begin{aligned}
	& \frac{\partial u^i_{R}}{\partial t} - \nabla\cdot(a(y)\nabla u^i_{R}) = 0 & \quad &\text{in } K_{R}\times (0,T] \\
	& u^i_{R} = 0 &  \quad &\text{on }  \partial K_{R} \times (0,T]\\
	& u^i_{R}(y,0) = \nabla \cdot (a(y)\mathbf{e}_i) & \quad &\text{in } K_{R}.
	\end{aligned}
	\right. 
\end{equation}
Then, the homogenized coefficient is approximated by
	\begin{equation}
	\label{eq:new a0 filtered}
	a^{0,R,L,T}_{ij}  := \int_{K_L} \mathbf{e}_i \cdot a(y) \mathbf{e}_j \mu_{L}(y) \,dy - 2 \int_{0}^{T} \int_{ K_{L}}u^i_{R}(y,t) u^j_{R}(y,t) \mu_{L}(y) dy\,dt.
	\end{equation}
As it will be shown in Section \ref{sec:main}, the $T$ parameter is crucial to obtain an exponential convergence of the resonance error.

%% file: tex/modified_elliptic_approach.tex
The second approach that we propose can be viewed as adding a correction term to the elliptic cell problem \eqref{eq:micro_problem_bounded_domain} to reduce the boundary effect in the interior region $K_L$. The new elliptic cell problems are given by
\begin{equation}
\label{eq:new_elliptic_micro_problem}
\left\lbrace
\begin{aligned}
-\nabla\cdot\Big( a(y)\left( \nabla \chi^i_{R,T,N} + \mathbf{e}_i \right)\Big) + [e^{-A_N T} g^{i}](y) &= 0 & \quad  & \text{in } K_R \\
\chi^i_{R,T,N} &= 0 & \quad & \text{on } \partial K_R ,
\end{aligned}
\right.
\end{equation} 
where
\begin{align*}
[e^{-A_N T} g^{i}](y)  := \sum_{k=1}^{N} e^{-\lambda_k T} g_k^{j} \varphi_k(y), 
\end{align*}
and $\{ \lambda_k, \varphi_k \}_{k=1}^{N}$ are the first $N$ dominant eigenvalues and eigenfunctions of the operator $A:=-\nabla \cdot \left(a(\cdot) \nabla  \right)$, equipped with Dirichlet boundary conditions. Moreover, $g^i(y):= \nabla \cdot a(y)\mathbf{e}_i$, and $g^{i}_k := \langle g^{i}, \varphi_k \rangle_{L^2(K_R)}$. The homogenized coefficient can then be approximated by 
\begin{equation}\label{eq:new_elliptic_upscaling_formula}
b^{0,R,L,T,N}_{ij} = \int_{K_L} \left( a_{ij}(y) + \sum_{k=1}^{d} a_{ik}(y) \partial_{k} \chi^j_{R,T,N} (y) \right) \mu_{L}(y) \; dy. 
\end{equation} 
It is worth mentioning that the correction term $[e^{-A_N T} g^{i}](y)$ is an approximation to $[e^{-A T} g^{i}](y)$, which corresponds to the solution of the parabolic PDE \eqref{eq:parabolic dirichlet problem} at time $T$. However, due to the exponential decay of the semigroup $e^{-A T}$ with respect to the eigenvalues of the operator $A$, one can approximate this correction term with exponential accuracy by computing a few dominant eigenmodes of $A$, instead of solving the full parabolic PDE \eqref{eq:parabolic dirichlet problem}.

%% file: tex/equivalence_rescaled.tex
In this section, we give a proof of the equivalence between elliptic and parabolic equations, thus legitimating the use of \eqref{eq:new a0 filtered} and \eqref{eq:new_elliptic_upscaling_formula}, in place of \eqref{eq:upscaling_formula_bounded_domain}, as upscaling model.
In the statement of the results, we will refer to the space $\mathcal{M}(\alpha,\beta,\Omega)$, which consists of symmetric matrices $a \in [L^{\infty}(\Omega)]^{d\times d}$ such that ${\alpha |\zeta|^2 \le \zeta \cdot a(y) \zeta \le \beta |\zeta|^2}$, ${\forall \zeta \in \mathbb{R}^d}$, ${\text{ a.e. } y \in \Omega\subset \mathbb{R}^d}$. We will also use the notation 
\begin{equation*}
X_0(\R_+,\Omega) := \set{ v \in L^2 \left( \R_+ ; H^1_0 (\Omega) \right) , \partial_t v \in L^2 \left( \R_+ ; H^{-1} (\Omega)\right) }.
\end{equation*} 
	\begin{theorem}\label{thm:weak form convergence}
		Let $a \in \mathcal{M}(\alpha,\beta,K_R)$ and let $\nabla\cdot \left(a\mathbf{e}_k\right) \in L^2(K_R)$, for $k =1,\dots,d$. Let $u_R^k\in X_0(\R_+,K_R)$ be the unique weak solution of \eqref{eq:parabolic dirichlet problem} and $\psi^k_{R} \in H^1_0( K_{R} )$ be the unique weak solution of \eqref{eq:micro_problem_bounded_domain}. Then, for $1\le j,k \le d$ the following identities hold
		
			\begin{equation}
			\label{eq:result1 proposition}
			\psi^k_{R}(y) = \int_{0}^{+\infty} u^k_R(y,t) \, dt ,
			\end{equation}
			\begin{equation}
			\label{eq:result2 proposition}
			\frac{1}{2} \int_{K_{R}} \nabla \psi^k_{R}(y) \cdot a(y) \nabla \psi^j_{R}(y) \;  dy = \int_{0}^{+\infty} \int_{ K_{R}}u^k_{R}(y,t) u^j_{R}(y,t) \; dy\;dt. 
			\end{equation}
	\end{theorem}	
	\begin{proof}
	
    We reformulate problem \eqref{eq:parabolic dirichlet problem} as the abstract Cauchy problem in $L^2(K_{R})$
    
    \begin{equation*}
	\left \lbrace
	\begin{aligned}
	&\frac{d u^k_R}{dt} + A u^k_R = 0 \\
	&u^k_{R}(0) = g^k, \quad   g^k(y) = \nabla \cdot \left( a(y) \mathbf{e}_k \right) \text{ in } L^2(K_{R}). 
	\end{aligned}
	\right.
	\end{equation*}
	
Here, the operator $A: H_0^1(K_{R}) \to H^{-1}(K_{R})$ is defined as $Au := -\nabla \cdot \left( a\nabla u\right)$. Then, $u^k_{R}(t) = e^{-t A } g^k$. We know that $\sigma(A)$, the spectrum of $A$, is contained in an open sectorial domain $\alpha+S_{\omega}$, where $\alpha\in\R$, $\alpha>0$ and
		\[S_{\omega} = \set{ z \in \mathbb{C} : \abs{\text{arg } z } < \omega,\, 0<\omega <\frac{\pi}{2} }. \] 
		Then, the Dunford integral representation 
		\begin{equation*}
		e^{-tA} = \frac{1}{2\pi i}\int_{\Gamma} e^{-tz} (zI - A)^{-1}\,dz
		\end{equation*}
		holds, where $\Gamma$ is an infinite curve lying in $\rho(A) := \mathbb{C} \setminus \sigma(A) $ and surrounding $\sigma(A) $ counterclockwise. Then, integrating in time we obtain
		\begin{align*}
		\int_{0}^{+\infty} u^k_{R}(t) \, dt &= \int_{0}^{+\infty} \frac{1}{2 \pi i} \int_{\Gamma} e^{-tz} \left(z I - A \right)^{-1} g^k \,dz \, dt \\
		&= \frac{1}{2 \pi i} \int_{\Gamma} \int_{0}^{+\infty} e^{-tz} \, dt \left(z I - A \right)^{-1} g^k \,dz \\
		&= \frac{1}{2 \pi i} \int_{\Gamma} \frac{1}{z} \left(z I - A \right)^{-1} g^k \,dz = A^{-1} g^k.		
		\end{align*}
		
		The first equality is given by the Dunford integral formula. The second equality is obtained by Fubini's theorem. The third equality is true because the double integral is bounded, $\lim_{t \rightarrow + \infty} e^{-tz} = 0$ since $Re(z)>0$ on $\Gamma  $. The last equality follows from the fact that the function $f(z) = 1/z$ is holomorphic in the interior of $\alpha + S_{\omega}$. 
		Since $A$ is an isomorphism and ${\psi^k_{R}}$ is the weak solution of ${A\psi^k_{R} = g^k}$, we have that $A^{-1} g^{k} = \psi^k_R$ and \eqref{eq:result1 proposition} is proved. 
		
		To prove \eqref{eq:result2 proposition}, we write the weak formulation of \eqref{eq:micro_problem_bounded_domain} and choose ${\psi^j_{R} = \int_{0}^{+\infty} u^j_{R} \; dt}$ as test function:
		\begin{align*}
		\int_{K_R} \nabla \psi^j_{R} \cdot a(y) \nabla \psi^k_{R} \,dy
		&= \left( \nabla \cdot \left( a \mathbf{e}_k \right), \psi^j_{R} \right)_{L^2(K_R)} \\
		&= \int_{0}^{+\infty} \left( \nabla\cdot\left(a \mathbf{e}_k \right) , u^j_{R} \right)_{L^2(K_R)} \,dt.
		\end{align*}
		Using the semigroup property of $e^{-tA}$ and the self-adjointness of $A$ we obtain
		\[
		\int_{K_{R}}\nabla\psi^k_{R}(y) \cdot a(y)\nabla \psi^j_{R}(y) \;dy=  \int_{0}^{+\infty} \left(u^k_{R}(\cdot,t/2), u^j_{R}(\cdot,t/2) \right)_{L^2(K_R)} \;dt,
		\]
		and conclude the proof by the change of variable $t/2\mapsto t$.
	\end{proof}

By the symmetry of $a$ and the weak form of \eqref{eq:micro_problem_bounded_domain}, we can rewrite \eqref{eq:upscaling_formula_bounded_domain} as:
\begin{equation*}
a^{0,R}_{ij}  = \frac{ 1 }{ \abs{ K_{R} } } \int_{  K_{R } } \mathbf{e}_i \cdot a (y) \mathbf{e}_j \,dy - \frac{ 1 }{ \abs{ K_{R} } } \int_{  K_{R } } \nabla \psi^i_{R} ( y ) \cdot a (y)  \nabla \psi^j_{R} ( y ) \;dy.
\end{equation*}
Theorem \ref{thm:weak form convergence} provides an equivalent expression, based on the solutions $u^{i}_R$ of the parabolic cell problems \eqref{eq:parabolic dirichlet problem} over infinite time domain, for the second integral in the above expression. This result is summarized as a corollary below.
\begin{corollary}
	Let $a$ satisfy the assumptions of Theorem \ref{thm:weak form convergence}, $a^{0,R}$ be defined by  \eqref{eq:upscaling_formula_bounded_domain} and $a^{0,R,R,+\infty}$ be defined by \eqref{eq:new a0 filtered} with $\mu_R \in \mathbb{F}_0(K_R)$ (note that $L=R$).
	Then
	\begin{equation*}
	a^{0,R,R,+\infty} = a^{0,R}.
	\end{equation*}
	Hence, using the classical result stated in Section \ref{sec:intro}, there exist a constant $C>0$ independent of $R$ such that 
	\begin{equation*}
	\|a^{0,R,R,+\infty} - a^{0}\|_F \le \frac{C}{R}.
	\end{equation*}
\end{corollary}
From this analysis, we can immediately see that, when ${T = +\infty}$, the parabolic approach does not result in any gain in comparison to the standard cell-problem \eqref{eq:micro_problem_bounded_domain}, as the two strategies are equivalent and have first order convergence rates in $1/R$.

%% file: tex/Exponential_Convergence_Parabolic.tex
The following Theorem \ref{thm: modelling error parabolic correctors filter} shows that an exponential convergence rate for the boundary error can be attained when the parameter $T$ is sufficiently small and an appropriate filter is used, as fully proved in \cite{AAP18b}.
If the coefficients $a_{ij}(y)$ are $K$-periodic and \eqref{eq:parabolic dirichlet problem} is solved with periodic boundary conditions and integer $R$, then Theorem \ref{thm:weak form convergence} still holds true by substituing functions $\psi^i_R$ in \eqref{eq:result1 proposition} and \eqref{eq:result2 proposition} with $\chi^j$ defined in \eqref{eq:micro_problem_periodic_domain}. Thus, it is possible to find an equivalent formula for the exact homogenized coefficients, which is based on a parabolic model with periodic boundary conditions.
The proof of Theorem \ref{thm: modelling error parabolic correctors filter} is based on such an equivalence result and on a decomposition of the resonance error in several terms, respectively accounting for the averaging error of periodic functions, the boundary mismatch between the two problems and the truncation in time.
\begin{theorem}\label{thm: modelling error parabolic correctors filter}
	Let $a\in\mathcal{M}(\alpha,\beta,K_R)$ be $K$-periodic, $\nabla\cdot\left(a\mathbf{e}_i\right)\in L^2(K_R)$ for any $i=1,\dots,d$, $\mu_{L} \in \mathbb{F}_{q}(K_{L})$, for $0<L<\lfloor R \rfloor$, $R>1$ and $T>0$. Then
	\begin{equation*}
		\| a^{0,R,L,T} - a^0 \|_F 
		\le
		C \left(
		L^{-(q + 1)} + e^{-\alpha \pi^2 T} + 
	\left(\frac{R}{\sqrt{T}} + 1\right)^{d-1}
		\frac{e^{-c\frac{\abs{R-L}^2}{T}}}{\abs{R-L}}
		+
		\frac{T^{d+1}}{\abs{R-L}^{2d}}
		e^{-2c\frac{\abs{R-L}^2}{T}}
		\right),
	\end{equation*}
	where $C>0$ is a constant independent of $R,L,T$ and $c = 1/4\beta$.
	Moreover, the choices $	L = (1-k_{o}) R$ and $T  = k_T R$, with $0 < k_{o} < 1$ and $k_T = \frac{k_{o}}{\pi \sqrt{4 \beta \alpha}}$ result in the following convergence rate in terms of $R$:
	\begin{equation}\label{eq:final bound parabolic}
	\| a^{0,R,L,T} - a^0 \|_F \le 
	C 
	\left[
	R^{-(q+1)} + 
	\gamma(R) e^{-\zeta R} 
	\right],
	\end{equation}
	with
	\begin{equation*}
	\zeta = \frac{\pi k_{o}}{2\sqrt{\beta/\alpha}}\text{ and }\;
	\gamma(R) =
	1 +
	\frac{\left(\sqrt{R} + 1\right)^{d-1}}{R} +
	\frac{e^{-\frac{\pi k_{o}}{2\sqrt{\beta/\alpha}} R} }{R^{d-1}}. 
	\end{equation*}
\end{theorem} 
The term $L^{-(q+1)}$ is the averaging error induced by using a filter function $\mu_L \in \mathbb{F}_{q}(K_L)$, and it can be made arbitrarily small by taking higher values for $q$. The term $e^{-\alpha \pi^2 T}$ originates from using a finite $T$ for the parabolic cell problem \eqref{eq:parabolic dirichlet problem}. The remaining terms are the errors due to the boundary conditions, which decay exponentially provided $T < |R-L|^2$. Moreover, the quasi-optimal scaling of $L$ and $T$ in terms of $R$ are found by equating the exponents of the truncation and boundary errors.   
Note that bound \eqref{eq:final bound parabolic} is similar to the one obtained in \cite{Glo11}, except for the term $T^{-2}$ that accounts for the effect of using a biased model equation.

%% file: tex/Exponential_Convergence_Modified_Elliptic.tex
The boundary error associated with formula \eqref{eq:new_elliptic_upscaling_formula} relies on how the parameter $T$ is tuned, similarly to the parabolic case. If $T = +\infty$, then the term $e^{-A_N T} g^{j}$ vanishes, and the standard Dirichlet cell problem \eqref{eq:micro_problem_bounded_domain} is recovered. Hence, no improvement over the first order convergence rate will be observed. In the following theorem, we specify the precise exponential upper bound for the modified elliptic approach. The proof is based on the equality ${\lim\limits_{N\to\infty}\chi^i_{R,T,N} = \int_{0}^{T} u^i_R \,dt }$ and it is developed in \cite{AAP18c}.

\begin{theorem}\label{thm:Exponential_Rate_Modified_Elliptic}
	Let $a \in \mathcal{M}(\alpha,\beta,K_{R})$ be $K$-periodic, $\nabla\cdot\left(a\mathbf{e}_i\right)\in L^2(K_R)$ for any $i=1,\dots,d$, $\mu_{L} \in \mathbb{F}_{q}(K_{L})$, for $0<L<\lfloor R \rfloor$, $R>1$ and $T>0$. Then
\begin{equation*}
\| b^{0,R,L,T,N} -  a^{0} \|_F  
\le C \left(
\sqrt{T} L^{-(q+1)} 
+ e^{-\frac{\alpha \pi^2}{2} T} 
+ \dfrac{R^{d-1}T^{\frac{5-d}{2}}}{\left| R-L \right|^3} e^{-c \frac{\left| R-L \right|^2}{T}} 
+ \dfrac{R^2}{L^{d/2}} e^{-\frac{C_d N^{2/d} T}{R^2}} 
\right),
\end{equation*}
where $C>0$ is a constant independent of $R,L$ and $T$ but may depend on $a$ and $\mu$, $c =  1/4\beta$ and $c_d>0$ is a constant (independent of $R,T,L$) that may depend on the dimension $d$, $\alpha$ and $\beta$. 
Moreover, the choices $L = (1-k_{o}) R$, $T  = k_T R$ and $N = \left\lfloor{\left( \frac{\alpha \pi^2}{2 c_d} \right)^{d/2} R^{d}} \right\rfloor$, with $0 < k_{o} < 1$ and $k_T = \frac{k_{o}}{ \pi \sqrt{2 \beta \alpha}}$ result in the following convergence rate in terms of $R$:
\begin{align*}
\| b^{0,R,L,T,N} -  a^{0} \|_F  \leq C \left[  R^{-q-\frac{1}{2}} + \gamma(R) e^{-\zeta R} \right],
\end{align*}
with 
\begin{equation*}
\zeta = \frac{\pi k_{o}}{\sqrt{8 \beta/\alpha}} \text{ and }\;
\gamma(R) = R^{2-d/2} + 
\ R^{\frac{d-3}{2}} + 1.
\end{equation*}
\end{theorem}
The upper bounds in Theorem \ref{thm:Exponential_Rate_Modified_Elliptic} have a similar character to those in Theorem \ref{thm: modelling error parabolic correctors filter}, except the error $e^{-\frac{c_d N^{2/d} T}{R^2}}$ which comes from the spectral truncation. In particular, Theorem \ref{thm:Exponential_Rate_Modified_Elliptic} shows that an exponential convergence (for the spectral error) will be achieved if the number of modes scales as ${N = O(R^d)}$ and ${T = O(R)}$. In practice, small values of $R$, e.g. $R=10$, are preferred for simulations. This makes the elliptic approach very favourable both from a computation and accuracy point of view.  	

%% file: tex/numerical_validation.tex
Here, we show the results of numerical tests performed using a two-dimensional periodic tensor for validating the convergence rates of Theorems \ref{thm: modelling error parabolic correctors filter} and \ref{thm:Exponential_Rate_Modified_Elliptic} (note that both theorems assume periodicity of the coefficient). In particular, we consider the following $2\times 2$ tensor 
\begin{equation}\label{eq:tensor exp}
a(y) = 
\begin{pmatrix}
\left(3 + \frac{2\sqrt{17}}{8 \sin( 2\pi y_1)+9}\right)^{-1} & 0 \\
0 & \left( \frac{1}{20} + \frac{ 2\sqrt{17} }{ 8\cos( 2\pi y_2)+9 } \right)^{-1}
\end{pmatrix}.
\end{equation}
We compute a numerical approximation of $a^{0,R,L,T}$ and $ b^{0,R,L,T,N}$ for many values of $R$ and with the optimal values for $L$ and $T$ (as expressed in Theorems \ref{thm: modelling error parabolic correctors filter} and \ref{thm:Exponential_Rate_Modified_Elliptic}). The error between the numerical approximations and the exact value $a^0$ is then plotted against $R$, see Figure \ref{fig:mod_error_PoissonMS2D}. 
The reference value for $a^0$ is computed by solving the standard elliptic corrector problems \eqref{eq:micro_problem_periodic_domain} with $R=1$, periodic boundary conditions and by using formula \eqref{eq: upscaling formula}. Numerical approximations for the parabolic formulation \eqref{eq:parabolic dirichlet problem} are computed by a P1-Finite Elements discretization with meshsize $h = 1/100$ in space, while a Rosenbrock formula of order 2 with tolerance $tol=10^{-5}$ and adaptive stepping scheme is used in time. For the modified elliptic approach we used a meshsize $h = 1/160$ and $N=60$ eigenmodes for approximating the right-hand side $e^{-A_N T} g^{j}$, in disregard of the size $R$. 

The decay of the overall upscaling error is pictured in Figure \ref{fig:mod_error_PoissonMS2D}. In particular, for relatively low to moderate values of $R$, e.g.,\ $1<R\leq 10$, the exponentially decaying boundary error is negligible in comparison to the averaging error when $q=1$ and $q=3$.   

\begin{figure}[h!]
	\centering
	\input{img/tot_err_diffq_PoissonMS2D}
	\input{img/ModEllApproach_q1q3}
	\ref{named}
	\caption{Modelling error for the homogenization of the multiscale coefficients \eqref{eq:tensor exp}. Modelling parameters are $k_o = 1/2$ and $k_T = \frac{k_{o}}{\pi \sqrt{4 \beta \alpha}}$.}
	\label{fig:mod_error_PoissonMS2D}
\end{figure}
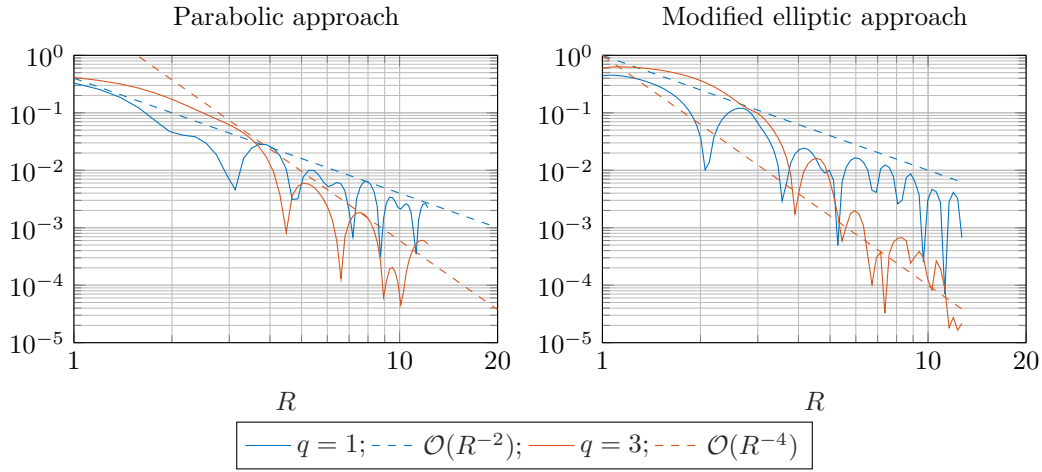

Besides the improved convergence rate, it is desirable that the computational cost of the proposed methods is comparable to the one of the classical model, that equals the cost of solving $d$ linear systems. Those can be solved by different numerical schemes, from LU decomposition, more suitable for smaller and two dimensional problems, to iterative methods like GMRES or CG, that are more indicated for large systems coming from three dimensional models. 
The first of the two proposed approaches, the parabolic model, can be efficiently solved by stabilized explicit ODE solver (such as RKC2 \cite{VHS90} or ROCK4 \cite{Abd02}), whose algorithms perform cheap matrix-vector multiplications iteratively. The number of iterations depends on the number of time steps and stages, but not on the dimension of the system. 
Lastly, the modified elliptic method has the same cost as the classical model, with the additional expense of accurately reconstructing the modified right-hand side by eigenfunction decomposition of the operator $A$. The eigenmodes computation can be done, for instance, by Krylov-Schur decomposition. 
A full analysis of the computational cost for the two methods will be addressed in future studies.

%% file: img/tot_err_diffq_PoissonMS2D.tex
%
%
\definecolor{mycolor1}{rgb}{0.00000,0.44706,0.74118}%
\definecolor{mycolor2}{rgb}{0.85098,0.32549,0.09804}%
\definecolor{mycolor3}{rgb}{0.46600,0.67400,0.18800}%
\definecolor{mycolor4}{rgb}{0.46667,0.67451,0.18824}%
\begin{tikzpicture}

\begin{axis}[%
width=.35\linewidth,
height=1.5in,
scale only axis,
xmode=log,
xmin=1,
xmax=20,
xtick = {1,10,20},
xticklabels={1, 10, 20},
minor xtick={1,2,3,4,5,6,7,8,9,10,20},
xlabel style={font=\color{white!15!black}},
xlabel={$R$},
ymode=log,
ymin=1e-05,
ymax=1,
ytick={ 1e-05, 0.0001,  0.001,   0.01,    0.1,      1},
yminorticks=true,
axis background/.style={fill=white},
title={Parabolic approach},
xmajorgrids,
xminorgrids,
ymajorgrids,
yminorgrids,
legend style={font=\color{white!15!black}},
legend columns = {-1},
legend entries = {$q=1$;, $\mathcal{O}(R^{-2})$;, $q=3$;, $\mathcal{O}(R^{-4})$},
legend to name = named
]
\addplot [color=mycolor1]
  table[row sep=crcr]{%
1	0.330626169724616\\
1.2	0.254397599946732\\
1.39	0.184897437031662\\
1.58	0.122579552786504\\
1.78	0.0725271626131855\\
1.97	0.047120242259227\\
2.16	0.0408937890574034\\
2.36	0.0383429264704232\\
2.55	0.0293849653459063\\
2.74	0.0186495499642463\\
2.94	0.0084941502282406\\
3.13	0.00459839651754817\\
3.32	0.0161838491801975\\
3.52	0.0245486987902064\\
3.71	0.0282272214189396\\
3.9	0.0279176671373027\\
4.1	0.0240322584309461\\
4.29	0.0179871671921167\\
4.49	0.0103811034070762\\
4.68	0.00311415743303332\\
4.87	0.0031787610853683\\
5.06	0.00763297784990334\\
5.26	0.00995004774363043\\
5.45	0.00994744888779441\\
5.64	0.00835108046793939\\
5.84	0.00620189542352524\\
6.03	0.00519184201080218\\
6.22	0.00555976915654399\\
6.42	0.00614289189216755\\
6.61	0.00580332404898095\\
6.8	0.00435730439472184\\
7	0.00194678745904454\\
7.19	0.000696093636504132\\
7.38	0.00318093576348139\\
7.58	0.00519075250444629\\
7.77	0.00626153692181883\\
7.96	0.00641787541156037\\
8.16	0.00565247519155928\\
8.35	0.00422278343910843\\
8.54	0.00237277788971828\\
8.74	0.000306708399504325\\
8.93	0.00144936018269495\\
9.12	0.00273442672248211\\
9.32	0.00339641400659792\\
9.51	0.00334909750573223\\
9.7	0.00282898656869318\\
9.9	0.00223602789711051\\
10.09	0.00211872353981891\\
10.28	0.00240313220746367\\
10.48	0.00261926296923823\\
10.67	0.00241687218689819\\
10.86	0.00177223838826022\\
11.06	0.000751119126602053\\
11.25	0.000347523326941479\\
11.44	0.00136970762437616\\
11.64	0.00218434908015978\\
11.83	0.00259670629087917\\
12.02	0.00261164129945081\\
12.22	0.00222108682628889\\
};

\addplot [color=mycolor1, dashed]
table[row sep=crcr]{%
	1	0.4\\
	20	0.001\\
};

\addplot [color=mycolor2]
  table[row sep=crcr]{%
1	0.412623682212494\\
1.2	0.36484430787491\\
1.39	0.317475988451356\\
1.58	0.269287980773917\\
1.78	0.220156342975966\\
1.97	0.178043740180867\\
2.16	0.142858210964592\\
2.36	0.114317321373973\\
2.55	0.0938694748904344\\
2.74	0.0782276076438085\\
2.94	0.0654415609847981\\
3.13	0.0546650459720177\\
3.32	0.0444797836894797\\
3.52	0.0342514361381116\\
3.71	0.0251603654696725\\
3.9	0.016941744641288\\
4.1	0.0095068657210257\\
4.29	0.00379519052112788\\
4.49	0.000823580858230318\\
4.68	0.00363445492762812\\
4.87	0.00526980751594491\\
5.06	0.00591483774014162\\
5.26	0.00576895517418422\\
5.45	0.005106789430847\\
5.64	0.00417816484393616\\
5.84	0.00312837842744531\\
6.03	0.00219268918047705\\
6.22	0.00135856772669628\\
6.42	0.000577171735907957\\
6.61	0.000130887393812718\\
6.8	0.000705725866410474\\
7	0.00122537219697269\\
7.19	0.00158377038743925\\
7.38	0.00179016932673204\\
7.58	0.00183852316850356\\
7.77	0.0017378889026882\\
7.96	0.001523387276908\\
8.16	0.00121475582585467\\
8.35	0.000885943019705818\\
8.54	0.000562466021560541\\
8.74	0.000264843761445098\\
8.93	6.1504477770947e-05\\
9.12	0.000118269364694177\\
9.32	0.000192129013019301\\
9.51	0.000202309990296092\\
9.7	0.000168446392577547\\
9.9	0.000105433975847851\\
10.09	4.58695722412878e-05\\
10.28	7.62934216141347e-05\\
10.48	0.000164596305444128\\
10.67	0.000258704654506013\\
10.86	0.000355435036806718\\
11.06	0.000451141906063479\\
11.25	0.000525701785197033\\
11.44	0.00057755304806179\\
11.64	0.000602965925709496\\
11.83	0.000598993580531919\\
12.02	0.000570828390376854\\
12.22	0.000519703608462209\\
};

\addplot [color=mycolor2, dashed]
  table[row sep=crcr]{%
1	6\\
20	3.75e-05\\
};

\end{axis}
\end{tikzpicture}%

%% file: img/ModEllApproach_q1q3.tex
%
%
\definecolor{mycolor1}{rgb}{0.00000,0.44706,0.74118}%
\definecolor{mycolor2}{rgb}{0.85098,0.32549,0.09804}%
\begin{tikzpicture}

\begin{axis}[%
width=.35\linewidth,
height=1.5in,
scale only axis,
xmode=log,
xmin=1,
xmax=20,
xtick = {1,10,20},
xticklabels={1, 10, 20},
minor xtick={1,2,3,4,5,6,7,8,9,10,20},
xlabel style={font=\color{white!15!black}},
xlabel={$R$},
ymode=log,
ymin=1e-05,
ymax=1,
yminorticks=true,
axis background/.style={fill=white},
title={Modified elliptic approach},
xmajorgrids,
xminorgrids,
ymajorgrids,
yminorgrids
]
\addplot [color=mycolor1]
  table[row sep=crcr]{%
1	0.440536317049787\\
1.03072240652209	0.450004534224021\\
1.06238867930668	0.452530862587029\\
1.09502781619681	0.449667702942637\\
1.128669705919	0.442762788987875\\
1.1633451554534	0.432906964984152\\
1.19908591824474	0.420701088988648\\
1.23592472327997	0.406656920841575\\
1.27389530505928	0.390986665714106\\
1.31303243448788	0.373890931389088\\
1.35337195071691	0.355402398986571\\
1.39495079396242	0.335659260799249\\
1.43780703933284	0.314684625939258\\
1.48197993169554	0.292513517174122\\
1.52750992161467	0.26924497252042\\
1.57443870239304	0.244995926970895\\
1.62280924825206	0.219976852127389\\
1.67266585368467	0.19423971709282\\
1.72405417401718	0.167956355554472\\
1.77702126721744	0.141393779568453\\
1.83161563698728	0.114601756701932\\
1.88788727717902	0.0879042328762961\\
1.94588771757639	0.0614156357072889\\
2.00567007108211	0.0354595423859881\\
2.06728908235507	0.0102361627686689\\
2.13080117794186	0.0139470441183656\\
2.19626451794833	0.0367838802510472\\
2.26373904929878	0.0579372542957475\\
2.33328656063125	0.0769921255721296\\
2.40497073887949	0.0934920889099582\\
2.47885722759307	0.106845409641916\\
2.5550136870494	0.116279853719805\\
2.63350985621243	0.120776029473353\\
2.71441761659491	0.119038887390665\\
2.79781105808265	0.109722026708325\\
2.88376654678106	0.0927534531322725\\
2.97236279494606	0.0722036291143319\\
3.06368093306352	0.0554849745314575\\
3.15780458414306	0.0434538130646888\\
3.25481994029442	0.0329538230350217\\
3.35481584165634	0.0227211345684462\\
3.45788385775044	0.0126141630842952\\
3.56411837133441	0.00293986481143105\\
3.67361666483139	0.00589049532568409\\
3.78647900941465	0.0134321132149595\\
3.90280875682923	0.0192453264163819\\
4.0227124340345	0.0229473565094008\\
4.14629984075437	0.0242776693697061\\
4.27368415002449	0.0231699708551524\\
4.40498201182854	0.0198964242143794\\
4.54031365991841	0.0152135448735822\\
4.67980302191621	0.0107490285447129\\
4.82357783279881	0.0089853684275838\\
4.97176975186899	0.0100140596806243\\
5.12451448332012	0.00671095761299412\\
5.281951900505	0.000495029977069905\\
5.44422617402243	0.00631937380769909\\
5.61148590373893	0.0115222948896911\\
5.78388425486656	0.0149814777579181\\
5.96157909822127	0.0164501313379987\\
6.1447331547904	0.0159743296269999\\
6.33351414474162	0.0137628487315525\\
6.52809494100976	0.0100798705530983\\
6.72865372760224	0.00446567822876831\\
6.93537416276799	0.00412530365542707\\
7.14844554717933	0.0105806375808964\\
7.36806299728077	0.0123909934765335\\
7.59442762396358	0.0112901396316537\\
7.82774671672956	0.00776172877231696\\
8.06823393351286	0.00260550457549991\\
8.31610949633354	0.0029837149850036\\
8.57160039296208	0.00748040245339621\\
8.83494058477954	0.00868534638581675\\
9.10637122102363	0.00637911114484424\\
9.38614085961695	0.00373484518716738\\
9.67450569477967	0.000252827976362169\\
9.97172979163494	0.00297453254078794\\
10.278085328022	0.00465797990404615\\
10.5938528437381	0.00424249652362918\\
10.9193214974386	0.00270757999640547\\
11.2547893314283	7.04282403402035e-05\\
11.6005635445889	0.00286529906989019\\
11.9569607736911	0.00412477095521535\\
12.324307383349	0.00322829022397708\\
12.7029397648835	0.000670296116673463\\
};

\addplot [color=mycolor1, dashed]
  table[row sep=crcr]{%
1	1\\
12.7029397648835	0.00619714306898267\\
};

\addplot [color=mycolor2]
  table[row sep=crcr]{%
1	0.599287533834436\\
1.03072240652209	0.615718737696122\\
1.06238867930668	0.625576628789971\\
1.09502781619681	0.630368838475232\\
1.128669705919	0.631460698457604\\
1.1633451554534	0.629844687895003\\
1.19908591824474	0.626153577931853\\
1.23592472327997	0.620763954880197\\
1.27389530505928	0.613889403820648\\
1.31303243448788	0.605647134634993\\
1.35337195071691	0.596096834201474\\
1.39495079396242	0.585265164389532\\
1.43780703933284	0.573158783361428\\
1.48197993169554	0.559777239150209\\
1.52750992161467	0.545126220780859\\
1.57443870239304	0.529225815646022\\
1.62280924825206	0.512106543260262\\
1.67266585368467	0.493802235686815\\
1.72405417401718	0.47435013238273\\
1.77702126721744	0.453794521936262\\
1.83161563698728	0.43219046269041\\
1.88788727717902	0.409607525226252\\
1.94588771757639	0.386133141173621\\
2.00567007108211	0.3618764012901\\
2.06728908235507	0.336972668690037\\
2.13080117794186	0.311589231091301\\
2.19626451794833	0.285933022523471\\
2.26373904929878	0.260260441178061\\
2.33328656063125	0.234890660170507\\
2.40497073887949	0.210222644522346\\
2.47885722759307	0.186755685261948\\
2.5550136870494	0.16510745648675\\
2.63350985621243	0.146008590006883\\
2.71441761659491	0.130220439122574\\
2.79781105808265	0.118234921542763\\
2.88376654678106	0.109469301089443\\
2.97236279494606	0.101074167609592\\
3.06368093306352	0.089618559509851\\
3.15780458414306	0.0758936976814535\\
3.25481994029442	0.062058724446503\\
3.35481584165634	0.0490086988238533\\
3.45788385775044	0.0370372413184487\\
3.56411837133441	0.0262557523830814\\
3.67361666483139	0.0167076528536254\\
3.78647900941465	0.00840050934852199\\
3.90280875682923	0.0017466870953016\\
4.0227124340345	0.00456868879850081\\
4.14629984075437	0.0092920735624652\\
4.27368415002449	0.0128612420009603\\
4.40498201182854	0.0152223168039916\\
4.54031365991841	0.0161840148272054\\
4.67980302191621	0.0153171914899888\\
4.82357783279881	0.0120403131905603\\
4.97176975186899	0.00705617960484639\\
5.12451448332012	0.0036681470802854\\
5.281951900505	0.00165136504524011\\
5.44422617402243	0.000588909696190326\\
5.61148590373893	0.00107043846719501\\
5.78388425486656	0.00178959560623171\\
5.96157909822127	0.001975705724095\\
6.1447331547904	0.0017015794044003\\
6.33351414474162	0.0010749773188955\\
6.52809494100976	0.000325667782145267\\
6.72865372760224	0.000103357803878157\\
6.93537416276799	0.000297144165640417\\
7.14844554717933	0.000377320436228845\\
7.36806299728077	3.26261116435356e-05\\
7.59442762396358	0.000349039246157368\\
7.82774671672956	0.000557763480375173\\
8.06823393351286	0.000656579432986248\\
8.31610949633354	0.000673388037797706\\
8.57160039296208	0.000590856834280729\\
8.83494058477954	0.000241598629381778\\
9.10637122102363	0.000310203148808585\\
9.38614085961695	0.000385339392899198\\
9.67450569477967	0.000303529048530149\\
9.97172979163494	0.000128164506404294\\
10.278085328022	8.34097463241479e-05\\
10.5938528437381	0.000265336496165003\\
10.9193214974386	0.000193131808366664\\
11.2547893314283	8.84515281792675e-05\\
11.6005635445889	1.83625205062432e-05\\
11.9569607736911	2.75678211215658e-05\\
12.324307383349	1.64978264313165e-05\\
12.7029397648835	2.16776745636249e-05\\
};

\addplot [color=mycolor2, dashed]
  table[row sep=crcr]{%
1	0.999999999999999\\
12.7029397648835	3.84045822174399e-05\\
};

\end{axis}
\end{tikzpicture}%

%% file: short_paper.bbl
\begin{thebibliography}{00}
\bibitem{Abd02}
A.~Abdulle, Fourth order Chebyshev methods with recurrence relation, SIAM J. Sci. Comput. 23(6) (2002), pp.2041--2054.
\bibitem{AAP18b}
A.~Abdulle, D.~Arjmand, E.~Paganoni, Reduction of the modelling error in numerical homogenization problems: a parabolic approach, Preprint, 2019.
\bibitem{AAP18c}
A.~Abdulle, D.~Arjmand, E.~Paganoni, A fully elliptic local problem with exponential decay for numerical homogenization problems, Preprint, 2019.
\bibitem{AEE12}
A.~Abdulle, W.~E, B.~Engquist, E.~Vanden-Eijnden, The heterogeneous multiscale method, Acta Numerica 21 (2012), pp. 1--87.
\bibitem{AR16}
D.~Arjmand, O.~Runborg, A time dependent approach for removing the cell boundary error in elliptic homogenization problems, J. Comput. Phys. 341 (2016), pp. 206--227.
\bibitem{BLP78}
A.~Bensoussan, J.~L.~Lions, G.~Papanicolaou, Asymptotic Analysis for Periodic Structures, North-Holland Publishing Co. (1978).
\bibitem{BlL10}
X.~Blanc, C.~Le Bris, Improving on computation of homogenized coefficients in the periodic and quasi-periodic settings, Netw. Heterog. Media 5 (1) (2010), pp. 1--29.
\bibitem{BoP04}
A.~Bourgeat, A.~Piatniski, Approximation of effective coefficients in stochastic homogenization, Ann. Ist. Henri Poincar\'e robab. Stat. 40 (2) (2004), pp. 153--165.
\bibitem{CiD99}
D.~Cioranescu, P.~Donato, An introduction to Homogenization, Oxford University Press (1999).
\bibitem{EMZ05}
W.~E, P.~Ming, P.~Zhang, Analysis of the heterogeneous multiscale method for elliptic homogenization problems, J. Amer. Math. Soc. 18(1) (2005), pp.121--156.
\bibitem{Glo11}
A.~Gloria, Reduction of the resonance error. Part 1: Approximation of homogenized coefficients, Math. Models Methods Appl. Sci. 21(8) (2011), pp.1601--1630.
\bibitem{GlO12}
A.~Gloria, F.~Otto, An optimal error estimate in stochastic homogenization of discrete elliptic equations, Ann. Appl. Probab. 22(1) (2012), pp.1--228.
\bibitem{JKO94}
V.~Jikov, S.~M.~Kozlov, O.~A.~Oleinik, Homogenization of Differential Operators and Integral Functionals, Springer-Verlag (1994).
\bibitem{Mou18}
J.~C.~Mourrat, Efficient methods for the estimation of homogenized coefficients. Found. Comp. Math., to appear.
\bibitem{YuE07}
X.~Yue, W.~E, The local microscale problem in the multiscale modeling of strongly heterogeneous media: effects of boundary conditions and cell size, J. Comput. Phys. 222(2) (2007), pp. 556--572.
\bibitem{VHS90}
J.~G.~Verwer, W.~H.~Hundsdorfer, B.~P.~Sommeijer, Convergence properties of the Runge-Kutta-Chebyshev method, Numer. Math. 57 (1990), pp. 157--178.
\end{thebibliography}
